\documentclass[12pt, dvipdfmx]{article}
\usepackage[mathscr]{eucal}
\usepackage{amssymb}
\usepackage{latexsym}
\usepackage{amsthm}
\usepackage{amsmath}
\usepackage[dvips]{graphicx}
\usepackage{psfrag}
\usepackage{a4wide}
\usepackage{tikz-cd}

\newtheorem{defn}{Definition}[section]
\newtheorem{thm}{Theorem}[section]

\newtheorem{lem}{Lemma}[section]
\newtheorem{rem}{Remark}[section]
\newtheorem{ex}{Example}[section]
\newtheorem{cor}{Corollary}[section]
\newtheorem{prop}{Proposition}[section]

\numberwithin{equation}{section}
\newcommand{\C}{{\mathbb C}}
\newcommand{\D}{{\mathbb D}}

\newcommand{\cM}{{\mathcal M}}
\newcommand{\cD}{{\mathcal D}}
\newcommand{\cS}{{\mathcal S}}
\newcommand{\cP}{{\mathcal P}}
\newcommand{\cQ}{{\mathcal Q}}
\newcommand{\cH}{{\mathcal H}}
\newcommand{\cK}{{\mathcal K}}
\newcommand{\ran}{\operatorname{ran}}
\newcommand{\la}{\langle}
\newcommand{\ra}{\rangle}
\newcommand{\p}{\partial}
\newcommand{\lam}{\lambda}
\newcommand{\aut}{\operatorname{Aut}}

\begin{document}

\title{Indefinite Schwarz-Pick inequalities on the bidisk\\
(application of the theory of analytic Hilbert modules)
}
\author{
{\sc Michio SETO}\\
[1ex]
{\small National Defense Academy,  
Yokosuka 239-8686, Japan} \\
{\small 
{\it E-mail address}: {\tt mseto@nda.ac.jp}}\\
}

\date{}

\maketitle
\begin{abstract}
Indefinite Schwarz-Pick inequalities for holomorphic self-maps of the bidisk are given 
as application of the spectral theory on Hilbert modules. 
\end{abstract}

\begin{center}
2010 Mathematical Subject Classification: Primary 46E22; Secondary 47B32\\
keywords: Schwarz-Pick inequality, Hilbert module, Hardy space
\end{center}

\section{Introduction}

The classical Schwarz-Pick inequality is 
fundamental in complex analysis and hyperbolic geometry, 
and also its functional analysis aspect has attracted a lot of interest.  
For example, Banach space theory related to the geometry derived from Schwarz-Pick inequality 
can be seen in Dineen~\cite{Dineen}. 
In Hilbert space operator theory, 
Schwarz-Pick inequalities for holomorphic functions of one and several variables 
were discussed by Anderson-Rovnyak~\cite{AR}, Anderson-Dritschel-Rovnyak~\cite{ADR}, 
Knese~\cite{Knese} and MacCluer-Stroethoff-Zhao~\cite{MSZ1,MSZ2} 
in the context of Pick interpolation, 
realization formula, de Branges-Rovnyak space and composition operator. 
Now, the purpose of this paper is to give 
some variants of Schwarz lemma and Schwarz-Pick inequality 
for the bidisk. 
Here the author would like to emphasize the following three points: 
\begin{enumerate}
\item we deal with holomorphic self-maps of the bidisk, 
\item our inequalities are indefinite in a certain sense, 
\item our method is based on the theory of analytic Hilbert modules. 
\end{enumerate}  

We shall introduce the language of the theory of Hilbert modules 
in the Hardy space over the bidisk. 
Let $\D$ be the open unit disk in the complex plane $\C$, 
$H^2$ be the Hardy space over the bidisk $\D^2$, 
and $H^{\infty}$ be the Banach algebra consisting 
of all bounded holomorphic functions on $\D^2$. 
Then $H^2$ is a Hilbert module over $H^{\infty}$, that is, 
$H^2$ is a Hilbert space invariant under multiplication of functions in $H^{\infty}$. 
A closed subspace $\cM$ of $H^2$ is called a submodule if 
$\cM$ is invariant under the module action. 
Comparing with the theory of the Hardy space over the unit disk $\D$,  
structure of submodules in $H^2$ is very complicated. 
However, there are some well-behaved classes of submodules in $H^2$. 
One of those classes was introduced by Izuchi, Nakazi and the author in \cite{INS}, 
and those members are said to be of INS type.    
In this paper, as an application of spectral theory on submodules of INS type, 
the following Schwarz-Pick type inequalities will be given 
(Theorem \ref{thm:5-1} and Theorem \ref{thm:5-2}): 
if $\psi=(\psi_1,\psi_2)$ is a holomorphic self-map on $\D^2$, then 
\[
0\leq d (\psi(z),\psi(w))\leq \sqrt{2}d (z,w)<\sqrt{2} \quad (z,w\in \D^2),
\]
where we set
\[
d(z,w)=\sqrt{\left|\frac{z_1-w_1}{1-\overline{w_1}z_1}\right|^2
+\left|\frac{z_2-w_2}{1-\overline{w_2}z_2}\right|^2
-\left|\frac{z_1-w_1}{1-\overline{w_1}z_1}\cdot\frac{z_2-w_2}{1-\overline{w_2}z_2}\right|^2} 
\]
for $z=(z_1,z_2)$ and $w=(w_1,w_2)$ in $\D^2$. 
Further, 
if $\psi$ belongs to a certain class defined in Section 2, 
then 
\[
0\leq d (\psi(z),\psi(w))\leq d (z,w)<1\quad (z,w\in \D^2).
\]

This paper contains four sections. 
Section 1 is this introduction. 
In Section 2, three classes of tuples of holomorphic functions on $\D^2$ are defined, 
and we show they are non-trivial. 
In Section 3, indefinite variants of Schwarz lemma are given 
with Hilbert space operator theory. 
In Section 4, 
as application of the theory of analytic Hilbert modules,  
indefinite variants of Schwarz-Pick inequality are given. 

\section{Schur-Drury-Agler class}

Let $k_{\lam}$ denote the reproducing kernel of $H^2$ at $\lam$ in $\D^2$, 
that is, 
\[
k_{\lam}(z)=\frac{1}{(1-\overline{\lam_1}z_1)(1-\overline{\lam_2}z_2)}
\quad (z=(z_1,z_2),\lam=(\lam_1,\lam_2)\in \D^2).
\]
Then we set 
\[
\cD=\left\{ \sum_{\lam} c_{\lam}k_{\lam}\ \mbox{(a finite sum)}: \lam\in \D^2, c_{\lam}\in \C \right\},
\]
the linear space generated by all reproducing kernels of $H^2$. 
We shall consider unbounded Toeplitz operators with symbols in $H^2$. 
Let $f$ be a function in $H^2$. 
Then $T_f$ denotes the multiplication operator of $f$, 
where we fix $\cD$ for the domain of $T_f$ .  
Then, since
\[
\la k_{\lam},T_fk_{\mu} \ra=
\la \overline{f(\lam)}k_{\lam},k_{\mu} \ra\quad (\lam,\mu \in \D^2),
\]
$T_f^{\ast}$ is defined on $\cD$ and 
\[
T_f^{\ast}k_{\lam}=\overline{f(\lam)}k_{\lam}\quad (\lam \in \D^2). 
\]

\begin{defn}\label{defn:2-1}
Let $m$ and $n$ be non-negative integers. 
We consider a tuple 
\[
\Phi_{m,n}=(\varphi_1,\ldots, \varphi_m, \varphi_{m+1},\varphi_{m+n})
\]
of $m+n$ holomorphic functions in $H^2$. 
Then $\cS(\D;m,n)$ denotes the set of all $\Phi_{m,n}$'s 
satisfying the following operator inequality on $\cD$:
\[0\leq 
\sum_{j=1}^mT_{\varphi_j}T_{\varphi_j}^{\ast}
-\sum_{k=m+1}^{m+n}T_{\varphi_k}T_{\varphi_k}^{\ast}\leq I.
\]
Equivalently, 
$\Phi_{m,n}$ belongs to $\cS(\D;m,n)$ if and only if
\[0\leq 
\dfrac{\sum_{j=1}^m\overline{\varphi_j(\lam)}\varphi_j(z)
-\sum_{k=m+1}^{m+n}\overline{\varphi_k(\lam)}\varphi_k(z)}
{(1-\overline{\lam_1}z_1)(1-\overline{\lam_2}z_2)}
\leq \frac{1}{(1-\overline{\lam_1}z_1)(1-\overline{\lam_2}z_2)}
\]
as kernel functions. 
\end{defn}
 
Since the author has been influenced by Drury~\cite{Drury}, 
in our paper,  
we would like to call $\cS(\D^2;m,n)$ a Schur-Drury-Agler calss of $\D^2$. 
Here two remarks are given. 
First, unbounded functions are not excluded from $\cS(\D^2;m,n)$
(cf. Definition 1 in Jury~\cite{Jury2} for the Drury-Arveson space). 
Throughout this paper, a triplet $(\varphi_1,\varphi_2,\varphi_3)$ 
consisting of functions in $H^{\infty}$ 
will be said to be bounded. 
Second, $\cS(\D^2;m,n)$ is more restricted than 
the class consisting of tuples of functions in $H^2$ satisfying the operator inequality
\[I- 
\sum_{j=1}^mT_{\varphi_j}T_{\varphi_j}^{\ast}
+\sum_{k=m+1}^{m+n}T_{\varphi_k}T_{\varphi_k}^{\ast}\geq 0. 
\]
In this paper, we will focus on the case where $m=2$ and $n=1$, that is,
\[
\cS(\D^2;2,1)=\{(\varphi_1,\varphi_2,\varphi_3)\in (\operatorname{Hol}(\D^2))^3:
0\leq T_{\varphi_1}T_{\varphi_1}^{\ast}+T_{\varphi_2}T_{\varphi_2}^{\ast}
-T_{\varphi_3}T_{\varphi_3}^{\ast}\leq I\}.
\]
This class is closely related to submodules of rank 3 
(see Wu-S-Yang~\cite{WSY} and Yang~\cite{Yang}). 
Further, we define other two classes as follows:
\begin{align*}
\cP(\D^2;2,1)&=\{(\varphi_1,\varphi_2,\varphi_3)\in (\operatorname{Hol}(\D^2))^3:
T_{\varphi_1}T_{\varphi_1}^{\ast}+T_{\varphi_2}T_{\varphi_2}^{\ast}
-T_{\varphi_3}T_{\varphi_3}^{\ast}\geq 0\},\\
\cQ(\D^2;2,1)&=\{(\varphi_1,\varphi_2,\varphi_3)\in (\operatorname{Hol}(\D^2))^3:
I-T_{\varphi_1}T_{\varphi_1}^{\ast}-T_{\varphi_2}T_{\varphi_2}^{\ast}
+T_{\varphi_3}T_{\varphi_3}^{\ast}\geq 0\}.
\end{align*}
Trivially, $\cP(\D^2;2,1)\cap \cQ(\D^2;2,1)=\cS(\D^2;2,1)$. 
First, we shall give examples of elements of  $\cS(\D^2;2,1)$. 

\begin{ex}\label{ex:2-1}\rm 
Let $\varphi_1=\varphi_1(z_1)$ 
and $\varphi_2=\varphi_2(z_2)$ be holomorphic functions of single variable. 
If $\|\varphi_1\|_{\infty}\leq 1$ and $\|\varphi_2\|_{\infty}\leq 1$, then 
$(\varphi_1,\varphi_2,\varphi_1\varphi_2)$ belongs to $\cS(\D^2;2,1)$.  
Indeed, since $T_{\varphi_1}$ and $T_{\varphi_2}$ are doubly commuting contractions, 
\[
I-T_{\varphi_1}T_{\varphi_1}^{\ast}-T_{\varphi_2}T_{\varphi_2}^{\ast}
+T_{\varphi_1\varphi_2}T_{\varphi_1\varphi_2}^{\ast}
=(I-T_{\varphi_1}T_{\varphi_1}^{\ast})(I-T_{\varphi_2}T_{\varphi_2}^{\ast})\geq 0,
\]
and 
\[
T_{\varphi_1}T_{\varphi_1}^{\ast}+T_{\varphi_2}T_{\varphi_2}^{\ast}
-T_{\varphi_1\varphi_2}T_{\varphi_1\varphi_2}^{\ast}
=T_{\varphi_1}T_{\varphi_1}^{\ast}
+T_{\varphi_2}(I-T_{\varphi_1}T_{\varphi_1}^{\ast})T_{\varphi_2}^{\ast}\geq 0.
\]
In particular, $(z_1,z_2,z_1z_2)$ belongs to $\cS(\D^2;2,1)$ and 
\[
T_{z_1}T_{z_1}^{\ast}+T_{z_2}T_{z_2}^{\ast}-T_{z_1z_2}T_{z_1z_2}^{\ast}
\]
is the orthogonal projection of $H^2$ onto the submodule generated by $z_1$ and $z_2$.   
\end{ex}

\begin{ex}\label{ex:2-2}\rm 
Let $\psi(z)=(\psi_1(z),\psi_2(z))$ be a holomorphic self-map of $\D^2$. 
Then, trivially, 
$\ran T_{\psi_1\psi_2/\sqrt{2}}$ is a subspace of
$\ran T_{\psi_1}$. 
Hence, by the Douglas range inclusion theorem and $\|T_{\psi_j}\|\leq 1$, we have
\[
0\leq 
T_{\psi_1\psi_2/\sqrt{2}}T_{\psi_1\psi_2/\sqrt{2}}^{\ast}
\leq \frac{1}{2}T_{\psi_1}T_{\psi_1}^{\ast}
\leq T_{\psi_1}T_{\psi_1}^{\ast}+T_{\psi_2}T_{\psi_2}^{\ast}\leq 2I. 
\]
Therefore, we have
\begin{align*}
0
&\leq \dfrac{1}{2}( T_{\psi_1}T_{\psi_1}^{\ast}+T_{\psi_2}T_{\psi_2}^{\ast}
-T_{\psi_1\psi_2/\sqrt{2}}T_{\psi_1\psi_2/\sqrt{2}}^{\ast})\\
&=T_{\psi_1/\sqrt{2}}T_{\psi_1/\sqrt{2}}^{\ast}+T_{\psi_2/\sqrt{2}}T_{\psi_2/\sqrt{2}}^{\ast}
-T_{\psi_1\psi_2/2}T_{\psi_1\psi_2/2}^{\ast}\\ 
&\leq T_{\psi_1/\sqrt{2}}T_{\psi_1/\sqrt{2}}^{\ast}+T_{\psi_2/\sqrt{2}}T_{\psi_2/\sqrt{2}}^{\ast}\\
&\leq I.
\end{align*}
Thus 
$(\psi_1/\sqrt{2},\psi_2/\sqrt{2},\psi_1\psi_2/2)$ belongs to $\cS(\D^2;2,1)$ 
for any holomorphic self-map $(\psi_1,\psi_2)$ of $\D^2$.
\end{ex}

\begin{ex}\rm
Further non-trivial examples of elements in $\cS(\D^2;2,1)$ 
related to the theory of Hilbert modules in $H^2$ 
can be obtained from Theorem 3.3 in Wu-S-Yang~\cite{WSY}.
\end{ex}

$\cP(\D^2;2,1)$ and $\cQ(\D^2;2,1)$ are closed 
under composition of elements in $\cQ(\D^2;2,1)$ 
in the following sense (cf. Theorem 2 in Jury~\cite{Jury2}). 

\begin{thm}\label{thm:2-1}
Let $(\varphi_1,\varphi_2,\varphi_3)$ be a triplet in 
$\cP(\D^2;2,1)$ {\rm(}resp. $\cQ(\D^2;2,1)${\rm)}, 
and $\psi=(\psi_1,\psi_2)$ be a holomorphic self-map of $\D^2$. 
If $(\psi_1,\psi_2,\psi_1\psi_2)$ belongs to $\cQ(\D^2;2,1)$,  
then 
$(\varphi_1\circ \psi, \varphi_2\circ \psi, \varphi_3\circ \psi)$ belongs to 
$\cP(\D^2;2,1)$ {\rm(}resp. $\cQ(\D^2;2,1)${\rm)}. 
\end{thm}

\begin{proof}
We set
\[
\Phi (z,\lam)=
\overline{\varphi_1(\lam)}\varphi_1(z)
+\overline{\varphi_2(\lam)}\varphi_2(z)
-\overline{\varphi_3(\lam)}\varphi_3(z).
\]
If $(\varphi_1,\varphi_2,\varphi_3)$ belongs to $\cP(\D^2;2,1)$, 
then, for any $\lam_1,\ldots ,\lam_n$ in $\D^2$,  we have
\begin{align*}
&\quad\ \la (T_{\varphi_1\circ \psi}T_{\varphi_1\circ \psi}^{\ast}
+T_{\varphi_2\circ \psi}T_{\varphi_2\circ \psi}^{\ast}
-T_{\varphi_3\circ \psi}T_{\varphi_3\circ \psi}^{\ast})
\sum_{i=1}^nc_ik_{\lam_i}, \sum_{j=1}^nc_jk_{\lam_j} \ra\\
&=\sum_{i,j=1}^nc_i\overline{c_j}\Phi(\psi(\lam_j),\psi(\lam_i))\la k_{\lam_i},k_{\lam_j} \ra\\
&=\sum_{i,j=1}^nc_i\overline{c_j}\Phi(\psi(\lam_j),\psi(\lam_i)) \la k_{\psi(\lam_i)},k_{\psi(\lam_j)} \ra
\frac{\la k_{\lam_i},k_{\lam_j} \ra}{\la k_{\psi(\lam_i)},k_{\psi(\lam_j)} \ra}\\
&=\sum_{i,j=1}^nc_i\overline{c_j}
\la (T_{\varphi_1}T_{\varphi_1}^{\ast}+T_{\varphi_2}T_{\varphi_2}^{\ast}
-T_{\varphi_3}T_{\varphi_3}^{\ast})k_{\psi(\lam_i)}, k_{\psi(\lam_j)} \ra
\frac{\la k_{\lam_i},k_{\lam_j} \ra}{\la k_{\psi(\lam_i)},k_{\psi(\lam_j)} \ra}.
\end{align*}
By the definition of $\cQ(\D^2;2,1)$ and Schur's theorem, we have 
\[
T_{\varphi_1\circ \psi}T_{\varphi_1\circ \psi}^{\ast}
+T_{\varphi_2\circ \psi}T_{\varphi_2\circ \psi}^{\ast}
-T_{\varphi_3\circ \psi}T_{\varphi_3\circ \psi}^{\ast}\geq 0.
\] 
Hence,  
$(\varphi_1\circ \psi, \varphi_2\circ \psi, \varphi_3\circ \psi)$ belongs to 
$\cP(\D^2;2,1)$. 
Similarly, 
considering $1-\Phi$, 
we have the statement on $\cQ(\D^2;2,1)$.  
\end{proof}

\begin{cor}\label{cor:2-1}
Let $(\varphi_1,\varphi_2,\varphi_3)$ be a triplet in $\cS(\D^2;2,1)$, 
and let $\psi=(\psi_1,\psi_2)$ be a holomorphic self-map of $\D^2$. 
If $(\psi_1,\psi_2,\psi_1\psi_2)$ belongs to $\cQ(\D^2;2,1)$, 
then $(\varphi_1\circ \psi, 
\varphi_2\circ \psi, \varphi_3\circ \psi)$ belongs to $\cS(\D^2;2,1)$.
\end{cor}

\section{Indefinite Schwarz lemmas}

In this section, we shall give inequalities which can be seen as variants of Schwarz lemma. 
We need several lemmas.  
\begin{lem}\label{lem:3-1}
Let $T$ be a non-negative bounded linear operator, 
and $P$ be an orthogonal projection on a Hilbert space $\cH$. 
If there exists some constant $c>0$ such that $0\leq T\leq cP$, then we may take 
$c=\|T\|$. 
\end{lem}

\begin{proof}
By elementary theory of self-adjoint operators, we have the conclusion. 
\end{proof}

\begin{lem}\label{lem:3-2}
Let $(\varphi_1,\varphi_2,\varphi_3)$ be a bounded triplet in $\cP(\D^2;2,1)$. 
Then $\varphi_3$ belongs to $\varphi_1H^2+\varphi_2H^2$.
\end{lem}

\begin{proof}
It follows from the operator inequality
\[
T_{\varphi_3}T_{\varphi_3}^{\ast}\leq 
T_{\varphi_1}T_{\varphi_1}^{\ast}+T_{\varphi_2}T_{\varphi_2}^{\ast}
\]
that  
$\ran T_{\varphi_3}$ is a subspace of 
\[
\ran \sqrt{T_{\varphi_1}T_{\varphi_1}^{\ast}
+T_{\varphi_2}T_{\varphi_2}^{\ast}}=\ran T_{\varphi_1}+\ran T_{\varphi_2}
\]
(see 
Theorem 2.2 attributed to Crimmins in Fillmore-Williams~\cite{FW}
or Theorem 3.6 in Ando~\cite{Ando}).
This concludes the proof. 
\end{proof}

\begin{lem}\label{lem:3-4}
Let $(\varphi_1,\varphi_2,\varphi_3)$ be a bounded triplet in $\cP(\D^2;2,1)$. 
If $\varphi_1(0,0)=\varphi_2(0,0)=0$, then 
\[
0\leq |\varphi_1(z)|^2+|\varphi_2(z)|^2-|\varphi_3(z)|^2
\leq \|T\|(|z_1|^2+|z_2|^2-|z_1z_2|^2)
\]
for any $z=(z_1,z_2)$ in $\D^2$, 
where we set 
\[
T=T_{\varphi_1}T_{\varphi_1}^{\ast}+T_{\varphi_2}T_{\varphi_2}^{\ast}
-T_{\varphi_3}T_{\varphi_3}^{\ast}.
\]
\end{lem}

\begin{proof}
Suppose that 
$\varphi_1$, $\varphi_2$ and $\varphi_3$ are bounded and 
$\varphi_1(0,0)=\varphi_2(0,0)=0$. 
Then, it follows from Lemma \ref{lem:3-2} that $\varphi_3(0,0)=0$. 
Hence $\varphi_1,\varphi_2$ and $\varphi_3$ belong to the submodule $\cM_0=z_1H^2+z_2H^2$. 
Then we have 
\[
\ran (T_{\varphi_1}T_{\varphi_1}^{\ast}+T_{\varphi_2}T_{\varphi_2}^{\ast}
-T_{\varphi_3}T_{\varphi_3}^{\ast})\subseteq \cM_0. 
\]
Further, by elementary spectral theory, we have
\[
\ran (T_{\varphi_1}T_{\varphi_1}^{\ast}+T_{\varphi_2}T_{\varphi_2}^{\ast}
-T_{\varphi_3}T_{\varphi_3}^{\ast})^{1/2}
\subseteq \overline{\ran} (T_{\varphi_1}T_{\varphi_1}^{\ast}+T_{\varphi_2}T_{\varphi_2}^{\ast}
-T_{\varphi_3}T_{\varphi_3}^{\ast})\subseteq \overline{\cM_0}=\cM_0. 
\]
Hence, it follows from the Douglas range inclusion theorem
that there exists a constant $c>0$ such that
\[
0\leq T_{\varphi_1}T_{\varphi_1}^{\ast}+T_{\varphi_2}T_{\varphi_2}^{\ast}
-T_{\varphi_3}T_{\varphi_3}^{\ast}\leq cP_{\cM_0},
\] 
where $P_{\cM_0}$ denotes the orthogonal projection of $H^2$ onto $\cM_0$. 
By Lemma \ref{lem:3-1}, we may take $c=\|T\|$. 
Hence we have 
\[0\leq 
T_{\varphi_1}T_{\varphi_1}^{\ast}+T_{\varphi_2}T_{\varphi_2}^{\ast}
-T_{\varphi_3}T_{\varphi_3}^{\ast}
\leq \|T\|P_{\cM_0}
=\|T\|(T_{z_1}T_{z_1}^{\ast}+T_{z_2}T_{z_2}^{\ast}-T_{z_1z_2}T_{z_1z_2}^{\ast})
\]
by Example \ref{ex:2-1}. 
In particular, 
\begin{align*}
(|\varphi_1(\lam)|^2+|\varphi_2(\lam)|^2-|\varphi_3(\lam)|^2)k_{\lam}(\lam)
&=\la (T_{\varphi_1}T_{\varphi_1}^{\ast}+T_{\varphi_2}T_{\varphi_2}^{\ast}
-T_{\varphi_3}T_{\varphi_3}^{\ast})k_{\lam},k_{\lam} \ra\\
&\leq \la \|T\|(T_{z_1}T_{z_1}^{\ast}
+T_{z_2}T_{z_2}^{\ast}-T_{z_1z_2}T_{z_1z_2}^{\ast})k_{\lam},k_{\lam} \ra\\
&=\|T\|(|\lam_1|^2+|\lam_2|^2-|\lam_1\lam_2|^2)k_{\lam}(\lam)
\end{align*}
for any $\lam=(\lam_1,\lam_2)$ in $\D^2$. 
This concludes the proof. 
\end{proof}

\begin{lem}\label{lem:3-5}
If $\psi=(\psi_1,\psi_2)$ be a holomorphic self-map on $\D^2$, then 
$(\psi_1,\psi_2,\psi_1\psi_2)$ belongs to $\cP(\D^2;2,1)$. 
\end{lem}

\begin{proof}
Since $\|\psi_j\|_{\infty}\leq 1$ for $j=1,2$, we have 
\[
T_{\psi_1}T_{\psi_1}^{\ast}
+T_{\psi_2}T_{\psi_2}^{\ast}-T_{\psi_1\psi_2}T_{\psi_1\psi_2}^{\ast}
=
T_{\psi_1}T_{\psi_1}^{\ast}
+T_{\psi_2}(I-T_{\psi_1}T_{\psi_1}^{\ast})T_{\psi_2}^{\ast}\geq 0
\]
Hence $(\psi_1,\psi_2,\psi_1\psi_2)$ belongs to $\cP(\D^2;2,1)$. 
\end{proof}

The following are indefinite Schwarz lemmas for the bidisk. 
\begin{thm}\label{thm:3-1}
If $\psi=(\psi_1,\psi_2)$ be a holomorphic self-map on $\D^2$ and $\psi(0,0)=(0,0)$, then 
\[
0\leq |\psi_1(z)|^2+|\psi_2(z)|^2-|\psi_1(z)\psi_2(z)|^2
\leq \|T\|(|z_1|^2+|z_2|^2-|z_1z_2|^2)
\]
for any $z=(z_1,z_2)$ in $\D^2$. 
\end{thm}

\begin{proof}
By Lemma \ref{lem:3-4} and Lemma \ref{lem:3-5}, we have the conclusion.  
\end{proof}

\begin{prop}\label{lem:3-6}
Let $(\varphi_1,\varphi_2,\varphi_3)$ be a triplet in $\cS(\D^2;2,1)$. 
If $\varphi_1(0,0)=\varphi_2(0,0)=0$, then 
\[
0\leq |\varphi_1(z)|^2+|\varphi_2(z)|^2-|\varphi_3(z)|^2
\leq |z_1|^2+|z_2|^2-|z_1z_2|^2
\]
for any $z=(z_1,z_2)$ in $\D^2$. 
\end{prop}

\begin{proof}
If $(\varphi_1,\varphi_2,\varphi_3)$ is bounded, 
then we have the conclusion immediately by Lemma \ref{lem:3-4}.
Suppose that $(\varphi_1,\varphi_2,\varphi_3)$ is unbounded. 
Setting $\psi_r(z_1,z_2)=(rz_1,rz_2)$ for $0<r<1$, 
$(\varphi_1\circ \psi_r,\varphi_2\circ \psi_r,\varphi_3\circ \psi_r)$ belongs to $\cS(\D^2;2,1)$ 
by Corollary \ref{cor:2-1} and Example \ref{ex:2-1}.  Moreover, 
$\varphi_1\circ \psi_r$, $\varphi_2\circ \psi_r$ and $\varphi_3\circ \psi_r$ 
are bounded on $\D^2$, 
and $\varphi_1\circ \psi_r(0,0)=\varphi_2\circ \psi_r(0,0)=0$. 
Hence we have 
\begin{align*}
0&\leq |\varphi_1(rz)|^2+|\varphi_2(rz)|^2-|\varphi_3(rz)|^2\\
&=|\varphi_1\circ \psi_r(z)|^2+|\varphi_2\circ \psi_r(z)|^2-|\varphi_3\circ \psi_r(z)|^2\\
&\leq |z_1|^2+|z_2|^2-|z_1z_2|^2
\end{align*}
by Lemma \ref{lem:3-4}.
Letting $r$ tend to $1$, we have the conclusion for unbounded triplets. 
\end{proof}

The following is another indefinite Schwarz lemma for the bidisk. 
\begin{thm}\label{thm:3-2}
Let $\psi=(\psi_1,\psi_2)$ be a holomorphic self-map on $\D^2$. 
If $(\psi_1,\psi_2,\psi_1\psi_2)$ belongs to $\cQ(\D^2;2,1)$
and
$\psi(0,0)=(0,0)$, then 
$(\psi_1,\psi_2,\psi_1\psi_2)$ belongs to $\cS(\D^2;2,1)$ and 
\[
0\leq |\psi_1(z)|^2+|\psi_2(z)|^2-|\psi_1(z)\psi_2(z)|^2
\leq |z_1|^2+|z_2|^2-|z_1z_2|^2
\]
for any $z=(z_1,z_2)$ in $\D^2$. 
Moreover, if equality 
\[
|\psi_1(z)|^2+|\psi_2(z)|^2-|\psi_1(z)\psi_2(z)|^2
= |z_1|^2+|z_2|^2-|z_1z_2|^2
\]
holds on some open set, 
then 
$(\psi_1, \psi_2)=(e^{i\theta_1}z_1,e^{i\theta_2}z_2)$ or 
$\psi=(e^{i\theta_2}z_2,e^{i\theta_1}z_1)$. 
\end{thm}

\begin{proof}
First, 
by Lemma \ref{lem:3-5}, $(\psi_1,\psi_2,\psi_1\psi_2)$ belongs to $\cS(\D^2;2,1)$.   
Hence, we have the inequality by Theorem \ref{thm:3-1}. 
Next, we suppose that 
\[
|\psi_1(z)|^2+|\psi_2(z)|^2-|\psi_1(z)\psi_2(z)|^2
= |z_1|^2+|z_2|^2-|z_1z_2|^2
\]
on an open set $V$. 
Then, by the polarization (see p. 28 in Agler-McCarthy~\cite{AM} or p. 2762 in Knese~\cite{Knese}), 
we have 
\[
\overline{\psi_1(\lam)}\psi_1(z)+\overline{\psi_2(\lam)}\psi_2(z)
-\overline{\psi_1(\lam)\psi_2(\lam)}\psi_1(z)\psi_2(z)
= \overline{\lam_1}z_1+\overline{\lam_2}z_2-\overline{\lam_1\lam_2}z_1z_2
\]
on $\overline{V}\times V$,
and this identity can be extended to $\D^2\times \D^2$. 
Then, for $j=1,2$, we have
\[
\left| \dfrac{\p \psi_1}{\p z_j} \right|^2
+\left| \dfrac{\p \psi_2}{\p z_j} \right|^2
-\left| \dfrac{\p \psi_1\psi_2}{\p z_j} \right|^2
=\left| \dfrac{\p z_1}{\p z_j} \right|^2
+\left| \dfrac{\p z_2}{\p z_j} \right|^2
-\left| \dfrac{\p z_1z_2}{\p z_j} \right|^2.
\]
Hence we have
\begin{equation}\label{eq:3-1}
\left| \dfrac{\p \psi_1}{\p z_j}(0,0) \right|^2
+\left| \dfrac{\p \psi_2}{\p z_j}(0,0) \right|^2
=1.
\end{equation}
Similarly, we have
\begin{equation}\label{eq:3-2}
\left| \dfrac{\p^2 \psi_1}{\p z_j^2}(0,0) \right|^2
+\left| \dfrac{\p^2 \psi_2}{\p z_j^2}(0,0) \right|^2
-4\left| \dfrac{\p \psi_1}{\p z_j}(0,0)\dfrac{\p \psi_2}{\p z_j}(0,0)  \right|^2=0.
\end{equation}
It follows from (\ref{eq:3-1}) that 
\[
\|\psi_1\|^2+\|\psi_2\|^2\geq 
\left| \dfrac{\p \psi_1}{\p z_1}(0,0) \right|^2
+\left| \dfrac{\p \psi_1}{\p z_2}(0,0) \right|^2
+\left| \dfrac{\p \psi_2}{\p z_1}(0,0) \right|^2
+\left| \dfrac{\p \psi_2}{\p z_2}(0,0) \right|^2
=2.
\]
Hence, $\|\psi_1\|=1$ and $\|\psi_2\|=1$ and
\[
\psi_i=c_{i1}z_1+c_{i2}z_2\quad (|c_{i1}|^2+|c_{i2}|^2=1).
\] 
Further, by (\ref{eq:3-2}), we have
\[
\dfrac{\p \psi_1}{\p z_j}(0,0)\dfrac{\p \psi_2}{\p z_j}(0,0)=0,
\]
that is, $c_{1j}c_{2j}=0$. 
This concludes the proof. 
\end{proof}

\begin{cor}\label{cor:3-1}
Let $f$ be a holomorphic function on $\D^2$. 
If $\| f\|_{\infty}\leq 1$
and
$f(0,0)=0$, then 
\[
0\leq |f(z)|^2
\leq |z_1|^2+|z_2|^2-|z_1z_2|^2
\]
for any $z=(z_1,z_2)$ in $\D^2$. 
\end{cor}

\begin{proof}
Set $\psi=(\psi_1,\psi_2)=(f,0)$. 
Then $\psi$ is a holomorphic self-map, $\psi(0,0)=(0,0)$ 
and $(\psi_1,\psi_2,\psi_1\psi_2)=(f,0,0)$ belongs to $\cQ(\D^2;2,1)$.  
\end{proof}

In the next example, we shall see that 
Theorem \ref{thm:3-2} gives a criterion for membership in $\cS(\D^2;2,1)$. 
\begin{ex}\rm 
For $z=(z_1,z_2)$, 
we set
\[
\psi_1(z)=\dfrac{z_1+z_2}{2},\quad \psi_2(z)=\dfrac{z_1-z_2}{2}
\]
and $\psi=(\psi_1,\psi_2)$. Then $\psi$ is a holomorphic self-map on $\D^2$
and $\psi(0,0)=(0,0)$. 
However, $(\psi_1,\psi_2,\psi_1\psi_2)$ does not belong to $\cS(\D^2;2,1)$. 
Indeed, 
\begin{align*}
&\quad\ |z_1|^2+|z_2|^2-|z_1z_2|^2-
(|\psi_1(z)|^2+|\psi_2(z)|^2-|\psi_1(z)\psi_2(z)|^2)\\
&=|z_1|^2+|z_2|^2-|z_1z_2|^2-( |\dfrac{z_1+z_2}{2}|^2
+|\dfrac{z_1-z_2}{2}|^2+|\dfrac{z_1+z_2}{2}\cdot \dfrac{z_1-z_2}{2}|^2   )
\\
&=|z_1|^2+|z_2|^2-|z_1z_2|^2-\dfrac{1}{2}(|z_1|^2+|z_2|^2)-\dfrac{1}{16}|z_1^2-z_2^2|^2 \\
&\to -\dfrac{1}{16}|e^{2i\theta_1}-e^{2i\theta_2}|^2\quad (|z_1|,|z_2|\to 1).
\end{align*}
It follows from this calculation and Theorem \ref{thm:3-2} that 
$(\psi_1,\psi_2,\psi_1\psi_2)$ does not belong to $\cS(\D^2;2,1)$. 
\end{ex}

\begin{rem}\rm
Let $\psi=(\psi_1,\psi_2)$ be a holomorphic self-map on $\D^2$. 
If $(\psi_1,\psi_2,\psi_1\psi_2)$ belongs to $\cS(\D^2;2,1)$, 
then, the proof of Theorem 1 in Jury~\cite{Jury1} can be applied and 
we have that the composition operator $C_{\psi}:H^2\to H^2$ is bounded. 
As its corollary, the inequality in Theorem \ref{thm:3-2} is obtained.  
\end{rem}

\begin{rem}[Kre\u{\i}n space geometry and $\D^2$]\rm 
We introduce a Kre\u{\i}n space structure into $\C^3$ as follows:
\[
\la z,w \ra_{\cK}
=z_1\overline{w_1}+z_2\overline{w_2}-z_3\overline{w_3}\quad
(z=(z_1,z_2,z_3),w=(w_1,w_2,w_3)\in \C^3).
\] 
Let $\cK$ denote this Kre\u{\i}n space, 
and let $\Phi$ be the map defined as follows:
\[
\Phi:\D^2 \to \cK,\quad (z_1,z_2) \mapsto (z_1,z_2,z_1z_2).
\]
Moreover, we set 
\begin{align*}
\Omega
&=\{(z_1,z_2)\in \C^2:0\leq |z_1|^2+|z_2|^2-|z_1z_2|^2<1\}\\
&=\{ z\in \cK: 0\leq \la z,z \ra_{\cK}<1 \}.
\end{align*}
Then, since 
\[
|z_1|^2+|z_2|^2-|z_1z_2|^2=
1-(1-|z_1|^2)(1-|z_2|^2),
\]
$\D^2$ is the bounded connected component of $\Omega$, 
and $\partial \D^2$, the topological boundary of $\D^2$, is equal to 
the subset 
\[
\{(z_1,z_2)\in \C^2:|z_1|^2+|z_2|^2-|z_1z_2|^2=1\}=
\{ z\in \cK: \la z,z \ra_{\cK}=1 \}.
\]
\end{rem}

\section{Indefinite Schwarz-Pick inequality}

Let $q_1=q_1(z_1)$ and $q_2=q_2(z_2)$ be inner functions of single variable. 
Then 
\[
\cM=q_1H^2+q_2H^2
\] 
is a submodule of $H^2$. 
This submodule was introduced by Izuchi-Nakazi-S~\cite{INS}, 
and is said to be of INS-type. In this section, 
we shall give an application of spectral theory on submodules of INS type
\footnote{
I remember that Izuchi showed me a fax from Nakazi. 
In which, Nakazi posed a problem and wrote ``\textit{it will be fruitful}". 
After their preliminary work, 
the problem was solved, 
and now it is known as the main theorem of \cite{INS}. 
}. 
In the general theory of Hilbert modules in $H^2$, 
the core (defect) operator of a submodule $\cM$ in $H^2$ is defined as follows:
\[
\Delta_{\cM}
=P_{\cM}-T_{z_1}P_{\cM}T_{z_1}^{\ast}-T_{z_2}P_{\cM}T_{z_2}^{\ast}
+T_{z_1z_2}P_{\cM}T_{z_1z_2}^{\ast},
\]
where $P_{\cM}$ denotes the orthogonal projection of $H^2$ onto $\cM$. 
For a submodule of INS-type, 
it is known that
\[
\Delta_{\cM}
=q_1\otimes q_1+q_2 \otimes q_2- (q_1q_2)\otimes (q_1q_2),
\]
where $\otimes$ denotes the Schatten form. 
Core operators were introduced and studied 
by Guo-Yang~\cite{GY} and Yang~\cite{Yang} in detail, 
and which are devices connecting reproducing kernels and submodules. 
In particular, the following formula is useful: 
\begin{equation}\label{eq:4-1}
k_{\lam}(\Delta_{\cM}k_{\lam})=P_{\cM}k_{\lam}.
\end{equation}
By application of those facts, 
Lemma \ref{lem:3-4} is generalized as follows.

\begin{lem}\label{lem:4-1}
Let $\cM$ be a submodule of finite rank whose core operator 
has a representation 
\[
\Delta_{\cM}=\sum_{j=1}^{n+1}\eta_j\otimes \eta_j-\sum_{j=n+2}^{2n+1}\eta_j\otimes \eta_j.
\]
If
$(\varphi_1,\varphi_2,\varphi_3)$ is a bounded triplet in $\cP(\D^2;2,1)$,  
and $\varphi_1$ and $\varphi_2$ belong to $\cM$, 
then  
\[
0\leq |\varphi_1(z)|^2+|\varphi_2(z)|^2-|\varphi_3(z)|^2
\leq \|T\|\left(
\sum_{j=1}^{n+1}|\eta_j(z)|^2-\sum_{j=n+2}^{2n+1}|\eta_j(z)|^2\right)
\]
for any $z$ in $\D^2$, 
where we set 
\[
T=T_{\varphi_1}T_{\varphi_1}^{\ast}+T_{\varphi_2}T_{\varphi_2}^{\ast}
-T_{\varphi_3}T_{\varphi_3}^{\ast}. 
\]
In particular, 
if $\cM=q_1H^2+q_2H^2$ for 
inner functions $q_1=q_1(z_1)$ and $q_2=q_2(z_2)$ of single variable,  
then
\[
0\leq |\varphi_1(z)|^2+|\varphi_2(z)|^2-|\varphi_3(z)|^2
\leq \|T\|(|q_1(z_1)|^2+|q_2(z_2)|^2-|q_1(z_1)q_2(z_2)|^2)
\]
for any $z=(z_1,z_2)$ in $\D^2$.
\end{lem}

\begin{proof}
By the same argument as the first half of the proof of Lemma \ref{lem:3-4}, 
we have 
\[0\leq 
T_{\varphi_1}T_{\varphi_1}^{\ast}+T_{\varphi_2}T_{\varphi_2}^{\ast}
-T_{\varphi_3}T_{\varphi_3}^{\ast}
\leq \|T\|P_{\cM}.
\]
Then, for any $\lam=(\lam_1,\lam_2)$ in $\D^2$, we have 
\begin{align*}
(|\varphi_1(\lam)|^2+|\varphi_2(\lam)|^2-|\varphi_3(\lam)|^2)k_{\lam}(\lam)
&=\la  (T_{\varphi_1}T_{\varphi_1}^{\ast}+T_{\varphi_2}T_{\varphi_2}^{\ast}
-T_{\varphi_3}T_{\varphi_3}^{\ast})k_{\lam},k_{\lam} \ra\\
&\leq \la \|T\|P_{\cM}k_{\lam},k_{\lam} \ra\\
&=\|T\|\la k_{\lam}(\Delta_{\cM}k_{\lam}),k_{\lam} \ra\\
&=\|T\|\left\la k_{\lam}\left(\sum_{j=1}^{n+1}\eta_j\otimes \eta_j
-\sum_{j=n+2}^{2n+1}\eta_j\otimes \eta_j\right)k_{\lam},k_{\lam} \right\ra\\
&=\|T\|\left(\sum_{j=1}^{n+1}|\eta_j(z)|^2-\sum_{j=n+2}^{2n+1}|\eta_j(z)|^2\right)k_{\lam}(\lam)
\end{align*}
by (\ref{eq:4-1}).
This concludes the proof.
\end{proof}

For $z=(z_1,z_2)$ and $w=(w_1,w_2)$ in $\D^2$, we set
\[
b_{w_j}(z_j)=
\frac{z_j-w_j}{1-\overline{w_j}z_j}\quad (j=1,2).
\]
Then, we note that
\[
|b_{w_1}(z_1)|^2+|b_{w_2}(z_2)|^2
-|b_{w_1}(z_1)b_{w_2}(z_2)
|^2
=1- (1-|b_{w_1}(z_1)|^2)
(1-|b_{w_2}(z_2)|^2)>0.
\]
Hence
\[
d(z,w)=\sqrt{|b_{w_1}(z_1)|^2+|b_{w_2}(z_2)|^2
-|b_{w_1}(z_1)b_{w_2}(z_2)
|^2}
\]
is defined. 

\begin{thm}\label{thm:5-1}
Let $\psi=(\psi_1,\psi_2)$ be a holomorphic self-map on $\D^2$. 
Then, 
\[
0\leq d (\psi(z),\psi(w))\leq \sqrt{2}d (z,w)<\sqrt{2}
\]
for any $z$ and $w$ in $\D^2$.  
\end{thm}

\begin{proof}
For $z=(z_1,z_2)$ and $w=(w_1,w_2)$ in $\D^2$, we set
\[
\varphi_j(z)=b_{\psi_j(w)}(\psi_j(z))=\frac{\psi_j(z)-\psi_j(w)}{1-\overline{\psi_j(w)}\psi_j(z)}.
\]
Then, 
$(\varphi_1,\varphi_2)$ is a holomorphic self-map on $\D^2$, and   
$(\varphi_1,\varphi_2,\varphi_1\varphi_2)$ belongs to $\cP(\D^2;2,1)$ 
by Lemma \ref{lem:3-5}. 
It follows from $\varphi_1(w)=\varphi_2(w)=0$ 
that $\varphi_1$ and $\varphi_2$ belong to the submodule 
$b_{w_1}H^2+b_{w_2}H^2$. 
Hence, by Lemma \ref{lem:4-1}, 
we have 
\begin{align*}
0
&\leq |\varphi_1(z)|^2+|\varphi_2(z)|^2-|\varphi_1(z)\varphi_2(z)|^2\\
&\leq \|T\|(|b_{w_1}(z_1)|^2+|b_{w_2}(z_2)|^2-|b_{w_1}(z_1)b_{w_2}(z_2)|^2)\\
&\leq 2(|b_{w_1}(z_1)|^2+|b_{w_2}(z_2)|^2-|b_{w_1}(z_1)b_{w_2}(z_2)|^2)\\
&< 2.
\end{align*}
This concludes the proof.
\end{proof}

\begin{thm}\label{thm:5-2}
Let $\psi=(\psi_1,\psi_2)$ be a holomorphic self-map on $\D^2$. 
If $(\psi_1,\psi_2,\psi_1\psi_2)$ belongs to $\cQ(\D^2;2,1)$, then
\[
0\leq d (\psi(z),\psi(w))\leq d (z,w)<1
\]
for any $z$ and $w$ in $\D^2$. 
Moreover, 
if equality 
\[
d (\psi(z),\psi(w))= d (z,w)
\]
holds on some open set, then $\psi$ belongs to $\aut(\D^2)$. 
\end{thm}

\begin{proof}
In this proof, we shall use the same notations as those in the proof of Theorem \ref{thm:5-1}, 
that is,
we set $\varphi_j=b_{\psi_j(w)}\circ \psi$ for $j=1,2$. 
Then, $(\varphi_1,\varphi_2,\varphi_1\varphi_2)$ belongs to $\cS(\D^2;2,1)$ 
by Theorem \ref{thm:2-1}. 
Moreover, since $\varphi_1$ and $\varphi_2$ belong to the submodule 
$b_{w_1}H^2+b_{w_2}H^2$, applying Lemma \ref{lem:4-1}, 
we have 
\begin{align*}
0
&\leq |\varphi_1(z)|^2+|\varphi_2(z)|^2-|\varphi_1(z)\varphi_2(z)|^2\\
&\leq |b_{w_1}(z_1)|^2+|b_{w_2}(z_2)|^2-|b_{w_1}(z_1)b_{w_2}(z_2)|^2\\
&< 1.
\end{align*}
Thus we have the first half. 
Further, 
combining the standard proof of the Schwarz-Pick inequality with Theorem \ref{thm:3-2}, 
we have the second half. 
\end{proof}

\begin{cor}
Let $f$ be a holomorphic function on $\D^2$. 
If $\| f\|_{\infty}\leq 1$, then 
\[
0\leq \left|\dfrac{f(z)-f(w)}{1-\overline{f(w)}f(z)}\right|^2
\leq d(z,w)
\]
for any $z$ and $w$ in $\D^2$. 
\end{cor}

\begin{proof}
In the proof of Corollary \ref{cor:3-1}, we showed that $(f,0,0)$ belongs to $\cQ(\D^2;2,1)$.
\end{proof}

Although the next fact is known in more general context 
(for example, see Lemma 9.9 in Agler-McCarthy~\cite{AM}), 
it should be mentioned here.  
\begin{prop}\label{prop:5-1}
$d$ is a distance on $\D^2$. 
\end{prop}

\begin{proof}
We shall give a proof different from that of Lemma 9.9 in Agler-McCarthy~\cite{AM}. 
Let $z$ and $w$ be two points in $\D^2$. 
We denote $z=(z_1,z_2)$ and $w=(w_1,w_2)$. 
First, it is trivial that $d(z,w)=d(w,z)$ by the definition of $d$. 
Second, let $d_j(z_j,w_j)$ be the usual pseudo-hyperbolic distance 
between $z_j$ and $w_j$ in $\D$. 
Then we have
\begin{equation}\label{eq:5-1}
1-(d(z,w))^2=\{1-(d_1(z_1,w_1))^2\}\{1-(d_2(z_2,w_2))^2\}.
\end{equation}
Hence, if $d(z,w)=0$ then $d_j(z_j,w_j)=0$ for each $j=1,2$, 
that is, $z_1=w_1$ and $z_2=w_2$. 
Third, we shall show the triangle inequality. 
Since $d$ is invariant under the action of $\aut(\D^2)$, 
it suffices to show that 
\[
d(z,w)\leq d(z,0)+d(0,w).
\]
We set $|z_j|=r_j$ and $|w_j|=s_j$ for $j=1,2$. 
Then the inequality 
\begin{equation}\label{eq:5-2}
d_j(z_j,w_j)\leq \frac{r_j+s_j}{1+r_js_j}
\end{equation}
is well known, in fact, (\ref{eq:5-2}) is equivalent to the triangle inequality for $d_j$. 
Moreover we note that
\begin{equation}\label{eq:5-3}
1-(d(z,0))^2=1-(r_1^2+r_2^2-r_1^2r_2^2)=(1-r_1^2)(1-r_2^2).
\end{equation}
Then, it follows from (\ref{eq:5-1}), (\ref{eq:5-2}) and (\ref{eq:5-3}) that 
\begin{align*}
(d(z,w))^2&=1-\{1-(d_1(z_1,w_1))^2\}\{1-(d_2(z_2,w_2))^2\}\notag\\
&\leq 1-\left\{1- \left(\frac{r_1+s_1}{1+r_1s_1}\right)^2\right\}
\left\{1- \left(\frac{r_2+s_2}{1+r_2s_2}\right)^2\right\}\notag\\
&=1-\frac{(1-r_1^2)(1-s_1^2)(1-r_2^2)(1-s_2^2)}{(1+r_1s_1)^2(1+r_2s_2)^2}\notag\\
&=1-\frac{\{1-(d(z,0))^2\}\{1-(d(0,w))^2\}}{(1+r_1s_1)^2(1+r_2s_2)^2}.
\end{align*}
Hence, we have
\begin{align*}
&\quad\ (1+r_1s_1)^2(1+r_2s_2)^2
\{(d(z,0)+d(0,w))^2-(d(z,w))^2\}\\
&\geq (1+r_1s_1)^2(1+r_2s_2)^2\left\{(d(z,0)+d(0,w))^2
-\left(1-\frac{\{1-(d(z,0))^2\}\{1-(d(0,w))^2\}}{(1+r_1s_1)^2(1+r_2s_2)^2}\right)\right\}\\
&=(1+r_1s_1)^2(1+r_2s_2)^2\left\{(d(z,0)+d(0,w))^2-1\right\}+\{1-(d(z,0))^2\}
\{1-(d(0,w))^2\}\\
&\geq (d(z,0)+d(0,w))^2-1+\{1-(d(z,0))^2\}\{1-(d(0,w))^2\}\\
&=2d(z,0)d(0,w)+(d(z,0)d(0,w))^2\\
&\geq 0.
\end{align*}
Therefore we have 
\[
(d(z,0)+d(0,w))^2-(d(z,w))^2\geq 0.
\]
This concludes the proof. 
\end{proof}


\begin{thebibliography}{99}

\bibitem{AM}J. Agler and J. E. McCarthy, 
\textit{Pick interpolation and Hilbert function spaces}. 
Graduate Studies in Mathematics, 44. American Mathematical Society, Providence, RI, 2002.

\bibitem{AR}
J. M. Anderson and J. Rovnyak, 
\textit{On generalized Schwarz-Pick estimates}. 
Mathematika 53 (2006), no. 1, 161--168 (2007). 

\bibitem{ADR}
J. M. Anderson and M. A. Dritschel and J. Rovnyak,
\textit{Schwarz-Pick inequalities for the Schur-Agler class on the polydisk and unit ball}. 
Comput. Methods Funct. Theory 8 (2008), no. 1-2, 339--361. 

\bibitem{Ando}T. Ando, 
\textit{de Branges spaces and analytic operator functions}. Lecture notes, Sapporo, Japan, 1990. 

\bibitem{Dineen}
S. Dineen,
\textit{The Schwarz lemma}. 
Oxford Mathematical Monographs. Oxford Science Publications. The Clarendon Press, Oxford University Press, New York, 1989.

\bibitem{Drury}
S. W. Drury, 
\textit{Remarks on von Neumann's inequality}. 
Banach spaces, harmonic analysis, and probability theory (Storrs, Conn., 1980/1981), pp. 14--32, 
Lecture Notes in Math., 995, Springer, Berlin, 1983. 

\bibitem{FW}P. A. Fillmore and J. P. Williams, 
\textit{On operator ranges}, 
Advances in Math.\ 7 (1971), 254--281. 

\bibitem{GY}K. Guo and R. Yang, 
\textit{The core function of submodules over the bidisk}. 
Indiana Univ. Math. J. 53 (2004), no. 1, 205--222.

\bibitem{INS}K. Izuchi, T. Nakazi and M. Seto,
\textit{Backward shift invariant subspaces in the bidisc.} II. 
J. Operator Theory 51 (2004), no. 2, 361--376.

\bibitem{Jury1}M. Jury, 
\textit{Reproducing kernels, de Branges-Rovnyak spaces, and norms of weighted composition operators}. 
Proc. Amer. Math. Soc. 135 (2007), no. 11, 3669--3675.

\bibitem{Jury2}M. T. Jury, 
\textit{Norms and spectral radii of linear fractional composition operators on the ball}. 
J. Funct. Anal. 254 (2008), no. 9, 2387--2400. 

\bibitem{Knese}G. Knese, 
\textit{A Schwarz lemma on the polydisk}. 
Proc. Amer. Math. Soc. 135 (2007), no. 9, 2759--2768. 

\bibitem{MSZ1}
B. D. MacCluer, K. Stroethoff and R. Zhao, 
\textit{Generalized Schwarz-Pick estimates}. 
Proc. Amer. Math. Soc. 131 (2003), no. 2, 593--599.

\bibitem{MSZ2}
B. D. MacCluer, K. Stroethoff and R. Zhao, 
\textit{Schwarz-Pick type estimates}. 
Complex Var. Theory Appl. 48 (2003), no. 8, 711--730. 

\bibitem{WSY}Y. Wu, M. Seto and R. Yang, 
\textit{Kre\u{\i}n space representation and Lorentz groups of analytic Hilbert modules}.
Sci. China Math. 61 (2018), no. 4, 745--768. 

\bibitem{Yang}R. Yang, 
\textit{The core operator and congruent submodules}. 
J. Funct. Anal. 228 (2005), no. 2, 469--489. 

\end{thebibliography}
\end{document}